\documentclass[a4paper]{amsart}

\usepackage{amsthm}
\usepackage{amsfonts}
\usepackage{amsmath,amsthm,amssymb}
\usepackage{comment}
\usepackage{fancybox}
\usepackage{epsf}
\usepackage[abs]{overpic}
\usepackage{layout}

\newtheorem{theorem}{Theorem}[section]
\newtheorem{lemma}[theorem]{Lemma}
\newtheorem{proposition}[theorem]{Proposition}

\newtheorem{notation}[theorem]{Notation}
\theoremstyle{definition}
\newtheorem{definition}[theorem]{Definition}
\theoremstyle{remark}
\newtheorem{remark}[theorem]{Remark}

\newtheorem{claim}[theorem]{Claim}

\numberwithin{equation}{section}

\graphicspath{{figures/}}

\begin{document} 

\title[Goussarov-Polyak-Viro's $n$-equivalence and the pure virtual braid group]{Goussarov-Polyak-Viro's $n$-equivalence and the pure virtual braid group}

%\title[Stanford's equivalence and virtual string links]{Stanford's equivalence and virtual string links}

\author[Yuka Kotorii]{Yuka Kotorii} 
\address{
Mathematical Analysis Team, RIKEN Center for Advanced Intelligence Project (AIP) 
\\
1-4-1 Nihonbashi, Chuo-ku, Tokyo 103-0027, Japan}
\address{Department of Mathematics, Graduate School of Science, Osaka University\\
1-10 Machikaneyama Toyonaka Osaka 560-0043 Japan
}
\address{interdisciplinary Theoretical \& Mathematical Sciences Program (iTHEMS) RIKEN
\\
 2-1, Hirosawa, Wako, Saitama 351-0198, Japan}
\email{yuka.kotorii@riken.jp}

\begin{abstract} 
In the context of finite type invariants, Stanford introduced a family of equivalence relations on knots defined by the lower central series of the pure braid groups and characterized the finite type invariants in terms of the structure of the braid groups.
%It is known that this equivalence can topologically characterize the finite type invariants. 
It is known that this equivalence and Ohyama's equivalence defined by a local move are equivalent.
On the other hand, in the virtual knot theory, the concept of Ohyama's equivalence was extended by Goussarov-Polyak-Viro, which called an $n$-equivalence. 
In this paper we extend Stanford's equivalence to virtual knots and virtual string links by using the lower central series of the pure virtual braid group, and call it an $L_n$-equivalence.
We then prove that the $L_n$-equivalence is equal to the $n$-equivalence on virtual string links.
Moreover we directly prove that two virtual string links are not distinguished by any finite type invariants of degree $n-1$ if they are $L_n$-equivalent, for any positive integer $n$.

%A $C_n$-move is a family of local moves on knots and links, which gives a topological characterization of finite type invariants of knots. 
\end{abstract} 
\thanks{The author is partially supported by Grant-in-Aid for Young Scientists (B) (No. 16K17586), Japan Society for the Promotion of Science. This work was in part supported by RIKEN iTHEMS Program.  }

% \date{\today}

\maketitle

%%%%%%%%%%%%%%%%%%%%%%%%%%%%%%%%%%%%%%%%%%%%%%%%%%%%%%%%%%%%%%%%%%%%%%%%%%%%%%%%%%%%%%%%%%%%
\section{Introduction}
The theory of finite type invariants of knots and links was introduced by Vassiliev \cite{V} and Goussarov \cite{G1, G2} and developed by Birman-Lin \cite{BL}.
People studied a filtration on the set of all knots derived from finite type invariants.
Through studying finite type invariants, Ohyama introduced a family of local moves which is defined as local moves satisfying some property \cite{O} (also \cite{G1}).
A filtration derived from Ohyama's moves implied the filtration derived from finite type invariants.
However, it had been an open question whether it held the converse implication or not. 
 
After that, it solved by Goussarov \cite{G3, G4} and Habiro \cite{H0, H} independently by introducing theories of surgery along embedded graphs in 3-manifolds, called $Y$-graphs (or variation axes) by Goussarov and claspers by Habiro. 
%An $n$-variation equivalence (called $n$-equivalence in \cite{G4}) or $C_n$-equivalence for links is generated by n-variations \cite{G4} or $C_n$-moves \cite{H}, respectively. 
Goussarov \cite{G4} and Habiro \cite{H0, H} proved that a geometric filtration derived from $n$-variation equivalence generated by $Y$-graphs or $C_n$-equivalence by claspers and the algebraic one derived from finite type invariants are the same.
Therefore, the finite type invariants are given a topological characterization.
Moreover, Goussarov proved in \cite{G4} that for knots in $S^3$and string links, the n-variation equivalence coincides with the Ohyama' equivalence.

Stanford also studied a filtration derived from finite type invariants by using the lower central series of the pure braid groups \cite{S2, S3}.
He gave an equivalency of two filtrations by finite type invariant and his equivalence relation in \cite{S3}. 
%, called $H_n$-equivalence. 
Therefore the finite type invariant was characterized in terms of the structure of the braid groups.

%Stanford proved in \cite{S2} that two links are not distinguished by any finite type invariant of degree $n$ if one is obtained from the other by inserting an element of the $(n+1)$-th lower central series subgroup of the pure braid group. 
%two knots are not distinguished by any finite type invariant of degree $n$ if and only if they are related by a finite sequence of $C_n$-moves and ambient isotopies. 
%Moreover Stanford \cite{S3} translated Habiro's result for $C_n$-moves into the pure braid setting. 

On the other hand, a {\it $($long$)$ virtual knot} is defined by a (long) knot diagram with virtual crossings modulo Reidemeiser moves, introduced by Kauffman \cite{K}.
Goussarov-Polyak-Viro \cite{GPV} showed that the (long) virtual knot can be redefined as {\it Gauss diagram} and also gave the theory of finite type invariants on Gauss diagrams. 
They also defined an $n$-equivalence on (long) virtual knots as an extension of Ohyama's equivalence and mentioned that the value of a finite type invariant of degree less than or equal to $n$ depended only the $n$-equivalence class.

In this paper, we extend Stanford's equivalence to (long) virtual knots, called an $L_n$-equivalence.
%The $L_n$-moves generate the $L_n$-equivalence on (long) virtual knots.
We prove that $L_n$-equivalence coincides with $n$-equivalence on long virtual knots (Theorem~\ref{L_nandn-equ}).
Moreover, we  directly prove that, for any non-negative integer $n$, two $L_n$-equivalent long virtual knots are not distinguished by any finite type invariants of degree $n-1$ (Proposition~\ref{L_{n+1}V_n}).
These results are also established on virtual string links.

\section*{Acknowledgements} 

The author thanks Professor Kazuo Habiro for a lot of comments, discussions and suggestions.
The author also thanks Professor Vassily Manturov for comments and suggestions.
%The author also thanks Professor Sergey Melikhov for many comments and suggestions.
%The author is partially supported by Grant-in-Aid for Young Scientists (B) (No. 16K17586), Japan Society for the Promotion of Science.
%This work was in part supported by RIKEN iTHES Program.

%%%%%%%%%%%%%%%%%%%%%%%%%%%%%%%%%%%%%%%%%%%%%%%%%%%%%%%%%%%%%%%%%%%%%%%%%%%%%%%%%%%%%%%%%%%
\section{Gauss diagram}

A {\it Gauss diagram on $k$ strands} is an ordered oriented $k$ intervals with several oriented chords having disjoint endpoints and equipped with sign as in Figure \ref{gaussdiagram} (which is defined up to isotopy of intervals).  
Here, we call the chord an {\it arrow}.
The {\it trivial Gauss diagram} is a Gauss diagram without arrow. 

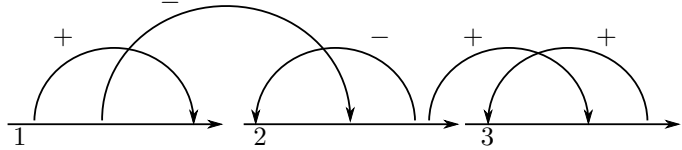
\begin{figure}[h]
  \begin{center}
\includegraphics[width=.7\linewidth]{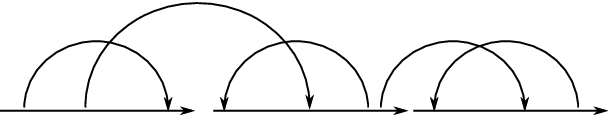}
\put(-235,33){$+$} \put(-195,46){$-$} \put(-120,33){ $-$} \put(-85,33){ $+$} \put(-35,33){ $+$}
\put(-250,-5){$1$} \put(-160,-5){$2$} \put(-75,-5){$3$}
    \caption{A Gauss diagram on 3 strands}
    \label{gaussdiagram}
  \end{center}
\end{figure}

{\it Reidemeister moves} among Gauss diagrams on several strands are the following three local moves in Figure \ref{Reidemeistermoves}:
First Reidemeister move (RI) is in the top row. Second Reidemeister move (RII) is in the second row.
Third Reidemeister move (RIII) is in the remaining two rows. 
%Here in RII and RIII. 
%we also consider the case of a Gauss diagram whose intervals are changed cyclically. 
\begin{figure}[h]
  \begin{center}
\includegraphics[width=.8\linewidth]{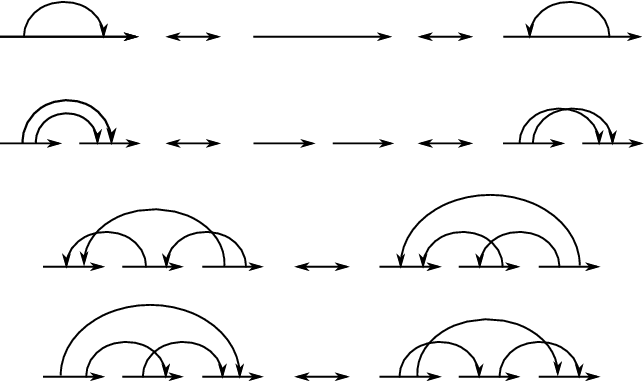}
\put(-207,160){$RI$} \put(-95,160){$RI$} \put(-207,112){$RII$} \put(-95,112){$RII$} \put(-150,58){$RIII$} \put(-150,8){$RIII$}
\put(-263,175){$\epsilon$} \put(-35,175){$\epsilon$} \put(-263,130){$-\epsilon$} \put(-263,115){$\epsilon$}\put(-50,125){$\epsilon$} \put(-20,125){$-\epsilon$}
\put(-263,65){$-$} \put(-233,80){$-$} \put(-190,70){$-$} \put(-100,65){$-$} \put(-80,85){$-$} \put(-60,70){$-$}
\put(-263,20){$+$} \put(-235,22){$+$} \put(-210,22){$+$} \put(-110,20){$+$} \put(-80,30){$+$} \put(-40,20){$+$} 
    \caption{The Reidemeister moves}
    \label{Reidemeistermoves}
  \end{center}
\end{figure}
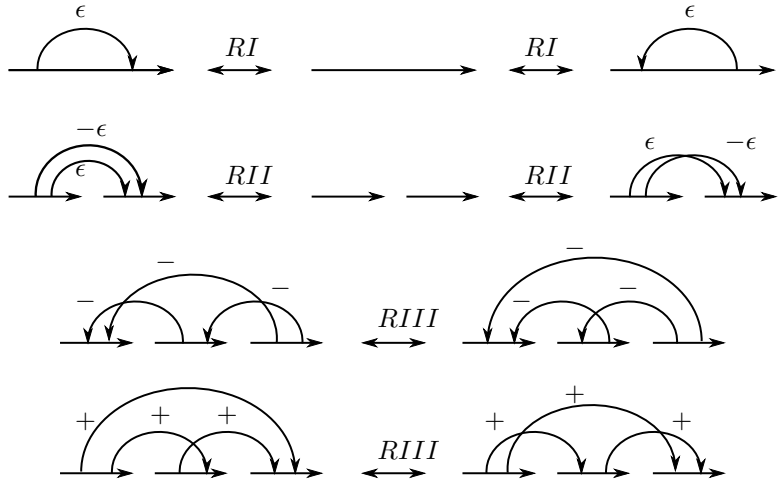

\begin{definition}
Two Gauss diagrams $D$ and $D'$ on several strands are said to be {\it equivalent} if $D$ and $D'$ are related by a sequence of Reidemeister moves. 
By $D {\sim} D'$ we mean that $D$ and $D'$ are equivalent.   
We define a {\it $k$-component virtual string link} to be the equivalence class of a Gauss diagram $D$ on $k$ strands, which is denoted by $[D]$.
In particular, 1-component virtual string link is called a {\it long virtual knot}.
We denote by $\mathcal{VSL}(k)$ the set of $k$-component virtual string links.
Similarly, the equivalence class of a Gauss diagram on a circle (or several circles) is a virtual knot (or virtual link, respectively).
\end{definition}

\begin{definition}
Let $D$ and $D'$ are two Gauss diagrams on the same strands.
We denote the {\it composition} of $D$ and $D'$ as $D \cdot D'$, which attaches a head of the $i$th interval of $D$ to an end of the $i$th interval of $D'$ for each $i$. 
\end{definition}
%%%%%%%%%%%%%%%%%%%%%%%%%%%%%%%%%%%%%%%%%%%%%%%%%%%%%%%%%%%%%%%%%%%%%%%%%%%%%%%%%%%%%%%%%%%
\section{Finite type invariant of virtual string links}

Goussarov Polyak and Viro defined a finite type invariant for (long) virtual knots in \cite{GPV}.
Similar way to classical knots, we can define Vassiliev-Goussarov filtration on $\mathbb{Z}$-module generated by the set of (long) virtual knots. 
Similarly, we can define them for virtual (string) links. 

A {\it dashed Gauss diagram} is a Gauss diagram with two types of signed chords, arrow and dashed arrow as in Figure \ref{dashedgaussdiagram}, possibly both with only arrows and with only dashed arrows.   
The dashed Gauss diagrams are said to be {\it equivalent} if they are related by a sequence of Reidemeister moves for arrows with fixing dashed arrows.
We also denote by $[D]$ the equivalence class of a dashed Gauss diagram $D$.
For each $n \geq 0$, let $\mathcal{SVSL}^{n}(k)$ denote the set of equivalence classes of dashed Gauss diagrams on $k$ strands with $n$ dashed arrows.
Then, in particular, $\mathcal{SVSL}^{0}(k)=\mathcal{VSL}(k)$.

\begin{figure}[h]
  \begin{center}
\includegraphics[width=.7\linewidth]{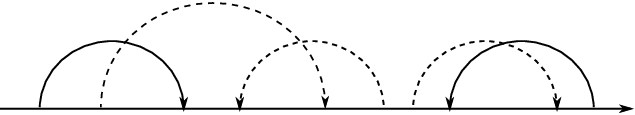}
\put(-225,30){$+$} \put(-185,46){$-$} \put(-120,30){ $-$} \put(-85,30){ $+$} \put(-35,30){ $+$}
    \caption{A dashed Gauss diagram on 1 strand}
    \label{dashedgaussdiagram}
  \end{center}
\end{figure}
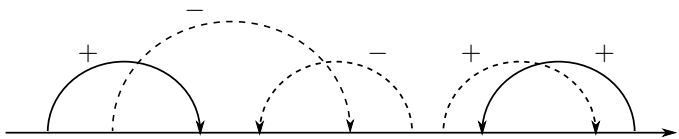

We construct a map $\varphi : \mathcal{SVSL}^{n}(k) \rightarrow \mathbb{Z}\mathcal{VSL}(k)$ as follows.
Let $D$ be a dashed Gauss diagram with $n$ dashed arrows.
Let $a_1, \cdots , a_n$ be the dashed arrows of $D$. 
For $\epsilon_1, \cdots , \epsilon_n$ in $\{\pm 1\}$, let $D_{\epsilon_1, \cdots , \epsilon_n}$ denote the Gauss diagram obtained from $D$ by replacing each dashed arrow $a_i$ with an arrow with the same sign if $\epsilon_i =1$ and removing each dashed arrow $a_i$ if  $\epsilon_i =-1$. 
We then define
$$\varphi([D]) = \sum_{\epsilon_1, \cdots , \epsilon_n \in \{\pm1\}} \epsilon_1 \cdot \dots \cdot \epsilon_n [D_{\epsilon_1, \cdots , \epsilon_n}].$$ 

\begin{definition}
Let $f$ be an invariant of $\mathcal{VSL}(k)$ with values in an abelian group $A$.
We extend it to $\mathbb{Z}\mathcal{VSL}(k)$ by linearly. Then $f$ is said to be a {\it finite type invariant of degree $n$} if $f \circ \varphi ([D])$ vanishes for any $k$-component dashed Gauss diagram $D$ with more than $n$ dashed arrows.    
\end{definition}

\begin{definition}
Denote by $J_n$ the subgroup of $\mathbb{Z}\mathcal{VSL}(k)$ generated by the set consisting of the element $\varphi([D])$, where $[D]$ is in $\mathcal{SVSL}^{n}(k)$.
It is easy to see that the $J_n(k)$'s form a descending filtration of two-sided ideals of the monoid ring $\mathbb{Z}\mathcal{VSL}(k)$ under the composition:
$$\mathbb{Z}\mathcal{VSL}(k)=J_0(k)  \supset J_1(k)  \supset J_2(k)  \supset \cdots , $$ 
which we call the Vassiliev-Goussarov filtration on $\mathbb{Z}\mathcal{VSL}(k)$.
\end{definition}

%Later, we will redefine $J_n(k)$ by using claspers.

%%two side ideal augmentation idealであること

\begin{remark}%%6.4(H)
Let $A$ be an abelian group and $n$ a positive integer.
The following two conditions are equivalent. 
A map is an $A$-valued finite type invariant of degree $n$ on $\mathcal{VSL}(k)$ and the map is a homomorphism of  abelian groups from $\mathbb{Z}\mathcal{VSL}(k)$ to $A$ which vanishes on $J_{n+1}(k)$
%%homとの対応
\end{remark}

\begin{definition}
For $n \geq 0$, two $k$-component virtual string links $L$ and $L'$ are said to be {\it $V_n$-equivalent} if $L$ and $L'$ are not distinguished by any finite type invariants of degree $n$ with values in any abelian group, equivalently, $L- L' \in J_{n+1}(k)$.  
%%V_\infについて
\end{definition}

%%%%%%%%%%%%%%%%%%%%%%%%%%%%%%%%%%%%%%%%%%%%%%%%%%%%%%%%%%%%%%%%%%%%%%%%%%%%%%%%%%%%%%%%%%%%%%%%%%%%%%%%%%%%%%%%%%%%%%%%%%%%%%%%%%
\section{Definition of $L_n$-equivalence}

By using the pure virtual braid group, we introduce a new equivalence relation on Gauss diagrams, called $L_n$-equivalence. 
Here it is known that the pure braid group is a subgroup of the pure virtual braid group (see \cite{F, M}).
% this is an extension of $H_n$-equivalence defined by Stanford \cite{S3}. 
We then give properties of the set of $L_n$-equivalence classes.

\begin{definition}[\cite{B, KL}]
A {\it pure virtual braid group} $PV_k$ on $k$ strands is a group represented by the following group representation.
\begin{align*}
PV_k = 
\left<
\begin{array}{l}
\hspace{-.8cm} \\
\hspace{-.8cm}
\end{array}
\right.
\left.
\begin{array}{l}
\hspace{-.3cm} \mu_{ij} \ (1 \leq i, j \leq k, i \neq j)  \hspace{-.2cm} 
\end{array}
\right.
%\left.
%\begin{array}{l}
%\mu_{ij} \ (1 \leq i, j \leq l, i \neq j)  
%\end{array}
%\right
\left|
\begin{array}{l}
\mu_{ij}\mu_{il}\mu_{jl} =\mu_{jl}\mu_{il}\mu_{ij}  \\
\mu_{ij}\mu_{lm} =\mu_{lm}\mu_{ij} 
\end{array}
\right.
\left.
\begin{array}{l}
\hspace{-.4cm} \text{(for all distinct }  i,j,l) \\
\hspace{-.4cm} \text{(for all distinct }  i,j,l,m)
%\hspace{-.2cm} (\{i,j\} \cap \{l,m\}= \emptyset)
\end{array}
\hspace{-.1cm} \right>.
\end{align*}
\end{definition}

Here, an element of the pure braid group is represented by a diagram as in Figure \ref{purevirtualbraid},
where $\mu_{ij}^\epsilon$ is correspondence with a horizontal arrow equipped with sign $\epsilon$ from the $i$-th strand to the $j$-th strand,    
and we determine that the orientation of the strand is from top to bottom.
For example, the diagram in Figure \ref{purevirtualbraid} represents $\mu_{12}\mu_{31}^{-1}\mu_{23}\mu_{12}^{-1} \in PV_3$.

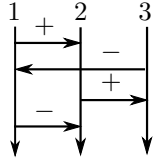
\begin{figure}[h]
  \begin{center}
\includegraphics[width=.15\linewidth]{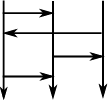}
\put(-55,53){$1$} \put(-30,53){$2$} \put(-6,53){$3$}
\put(-45,48){$+$} \put(-20,38){$-$} \put(-20,27){$+$} \put(-45,16){$-$}
    \caption{An element of the pure virtual braid group $PV_3$}
    \label{purevirtualbraid}
  \end{center}
\end{figure}

Let $h \in PV_k$ and $h' \in PV_{k'}$. 
%We denote by $h \otimes h' \in PV_{k+k'}$ the 横に並べる of $h$ and $h'$ in $PV_{k+k'}$.
We denote the {\it composition} and {\it tensor product} of $h$ and $h'$ as $h \cdot h' = {\includegraphics[width=.04\linewidth]{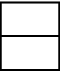}\put(-10,10){$h$} \put(-10,1){$h'$}} \in PV_{k}$ if $k=k'$ and $h \otimes h' = {\includegraphics[width=.08\linewidth]{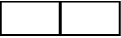} \put(-25,1){$h$} \put(-10,1){$h'$}} \in PV_{k+k'}$ for any $k$ and $k'$, respectively. 
By $\Gamma_n(G)$ we mean the $n$-th lower central subgroup of the group $G$, that is, $\Gamma_1(G)= G$ and $\Gamma_n(G)= [\Gamma_{n-1}(G), G]$, which is the commutator of $\Gamma_{n-1}(G)$ and $G$, that is, $<[a,b] \mid a \in \Gamma_{n-1}(G), b \in G >$  where $[a,b]=aba^{-1}b^{-1}$.

\begin{definition}
Two Gauss diagrams $D$ and $D'$ on several strands are related by  an {\it $L_n$-move} if there are a positive integer $k$, an element $h$ in the $n$-th lower central subgroup $\Gamma_n(PV_k)$ of the pure virtual braid group on $k$ strands and not in $\Gamma_{n+1}(PV_k)$, and an embedding $e$ of $k$ strands of $h$ such that $D^{(h,e)}=D'$, 
where $D^{(h,e)}$ is obtained from $D$ by attaching $h$ by  an embedding $e$ of $k$ strands of $h$ in the intervals of $D$ except for the endpoints of all arrows of $D$ as in Figure \ref{clasperdef}.
By $D \overset{L_n}{\rightarrow} D'$ we mean that $D'$ is obtained from $D$ by a $L_n$-move. 
In particular, we write $D \overset{(h,e)}{\rightarrow} D'$ if $D'=D^{(h,e)}$.

\begin{figure}[h]
  \begin{center}
\includegraphics[width=.8\linewidth]{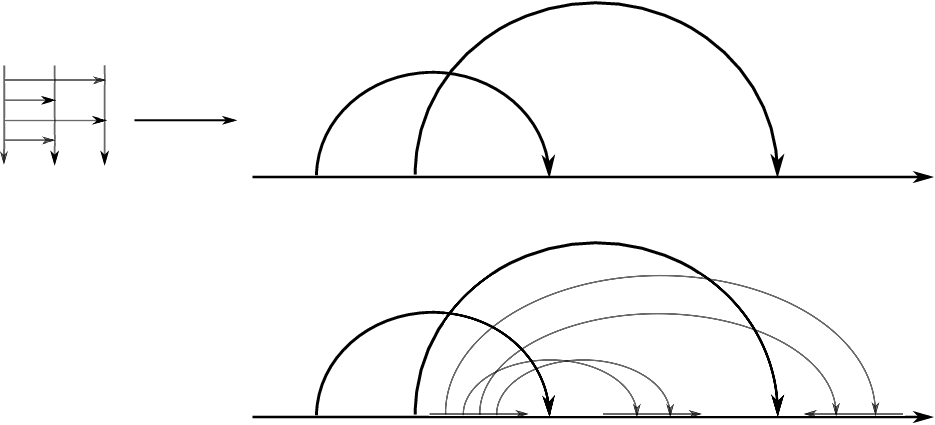}
\put(-235,98){$e$} \put(-275,70){$h$} \put(-0,73){$D$} \put(-145,-5){$1$} \put(-90,-5){$2$} \put(-30,-5){$3$} 
 \put(-245,5){$D^{(h,e)}=$}
 \put(-290,118){$1$} \put(-275,118){$2$} \put(-260,118){$3$}
  \put(-283,110){{\small $+$}} \put(-269,98){$-$} \put(-283,80){$+$} \put(-269,88){$-$} 
\put(-190,100){$+$} \put(-130,130){$+$}
\put(-190,25){$+$} \put(-130,55){$+$}
\put(-40,20){$-$} \put(-40,35){$+$} \put(-105,8){$-$} \put(-80,8){$+$}
    \caption{An embedded pure virtual braid $(h,e)$ for a Gauss diagram on 1 strand}
    \label{clasperdef}
  \end{center}
\end{figure}
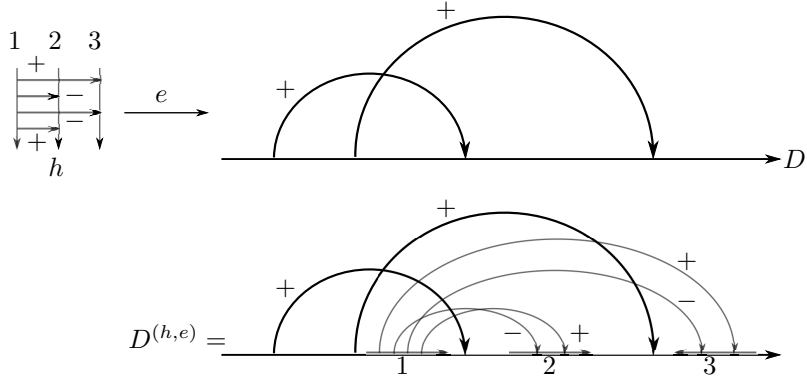

We call a pair $(h,e)$ for $D$ an {\it embedded pure virtual braid} for $D$.  
We define that a pair $(h,e)$ is of {\it degree} $n$ if $h \in \Gamma_n(PV_k)$ and $h \notin \Gamma_{n+1}(PV_k)$, 
where $k$ is a positive integer, and denote the degree of the pair $(h,e)$ by deg$(h,e)$.
%Hereinafter, we omit the number $k$ of strands if it is not important. 
Two embedded pure virtual braids for $D$ are {\it disjoint} if their embeddings are disjoint in the intervals of $D$.
For disjoint pairs $(h_1,e_1)$ and $(h_2,e_2)$ for $D$, $D^{\{(h_1,e_1), (h_2,e_2) \}}$ means $(D^{(h_1,e_1)})^{(h_2,e_2)}$ or equivalently $(D^{(h_2,e_2)})^{(h_1,e_1)}$.
Moreover, we can represent $D^{\{(h_1,e_1), (h_2,e_2) \}}$ as $D^{(h_1 \otimes h_2,e_1 \otimes e_2)}$.
Here it is easy to see that the degree of $(h_1 \otimes h_2,e_1 \otimes e_2)$  is $\min \{\text{deg}(h_1,e_1), \text{deg}(h_2,e_2) \}$. 
\end{definition} 

\begin{remark}
The first relation of the pure virtual braid group corresponds with the second relation of third Reidemeister moves  (illustrated in Figure \ref{Reidemeistermoves}). 
Therefore $[D^{(h,e)}]$ does not depend on a word representing a pure virtual braid $h$. 
On the other hand, we consider $h$ in $D^{(h,e)}$ as a word representing $h$.
For example, we consider $D^{(h \cdot h^{-1},e)}$ and $D$ are different Gauss words and $D^{(h \cdot h^{-1},e)}$ equals to $D$ up to a sequence of second Reidemeister moves. 
\end{remark}

In order to give properties of the $L_n$-move, we define the parallel embedding. 
Let $h$, $h' \in PV_k$ and $(h,e)$ and $(h',e')$ are disjoint embedded pure virtual braids for a Gauss diagram $D$.
The pair $(h,e)$ is {\it upper} (or {\it under}, respectively) {\it parallel} to $(h',e')$ (or embedding $e$ is upper (or under) parallel to $e'$) if for each $i$ $(1 \leq i \leq k)$, the  orientations of embeddings of $i$-th strands of $h$ and $h'$ by $e$ and $e'$ are the same, an $i$-th embedding by $e$ is upper (or under, respectively) than the $i$-th one by $e'$ with respect to the orientation, and there is no endpoints of arrows of $D$ between embedded $i$-th strands for each $i$, as in Figure \ref{parallel}. 
That is, $D^{\{(h,e), (h', e')\}}$ equals to $D^{(h \cdot h', e)}$ (or $D^{(h'\cdot h, e)}$, respectively). 
We note that the following statements are the same.
An embedded pure virtual braid $(h, e)$ is an under parallel to $(h', e')$ and $(h', e')$ is an upper parallel to $(h, e)$.

\begin{proposition}\label{}%%
The $L_n$-moves and Reidemeister moves generate an equivalence relation on Gauss diagrams.
\end{proposition}

\begin{proof}
The case of the reflexive and  transitive relation are obvious. We show the symmetric relation. Let $D'=D^{(h,e)}$ where $h \in \Gamma_n(PV_k)$ and $h \notin \Gamma_{n+1}(PV_k)$. 
Then 
$D'^{(h^{-1},e')} = D^{({h^{-1}} \cdot h, e)}$, where  ${(h^{-1},e')}$ is upper parallel to $(h,e)$.  
Since the Gauss diagram $D^{(h^{-1} \cdot h,e)}$ equals to $D$ up to a sequence of RII's, we have that $D' \overset{L_n}{\rightarrow} D$.
\end{proof}

We call this equivalence relation an {\it $L_n$-equivalence}. By $D \overset{L_n}{\sim} D'$ we mean that $D$ and $D'$ are $L_n$-equivalent.   

\begin{proposition}\label{movedegree}%%3.7(H)
If $1 \leq n \leq n'$, then an $L_{n'}$-move is achieved by an $L_n$-move. Therefore $L_{n'}$-equivalence implies $L_{n}$-equivalence.
\end{proposition}

\begin{proof}
By the property of the lower central series, $\Gamma_n(PV_k) < \Gamma_{n-1}(PV_k)$ for any $n \geq 2$.
For any embedded pure virtual braid $(h,e)$ of degree more than or equal to $n-1$ for $D$, there exists a set $H$ of disjoint embedded pure virtual braids for $D$ such that $D^{(h,e)}=D^{H}$ and each element in $H$ has degree $n-1$, because any element of $\Gamma_{n}$ is represented by the product of elements in $\Gamma_{n-1}$ each of which is not in $\Gamma_{n}$. 
\end{proof}

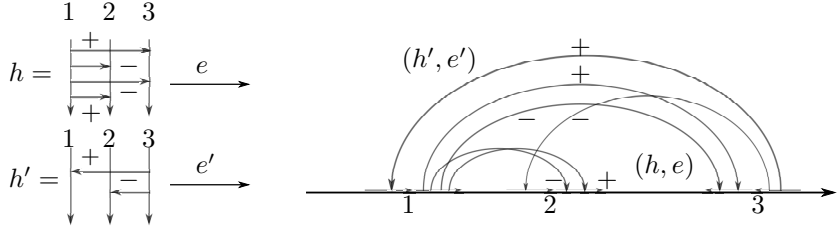
\begin{figure}[h]
  \begin{center}
\includegraphics[width=.8\linewidth]{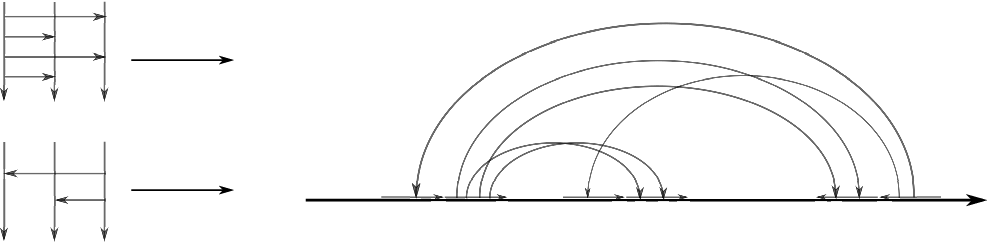}
\put(-240,58){$e$} \put(-310,55){$h=$} \put(-310,15){$h'=$} 
\put(-163,5){$1$} \put(-110,5){$2$} \put(-32,5){$3$} 
 \put(-240,20){$e'$}
 \put(-290,78){$1$} \put(-275,78){$2$} \put(-260,78){$3$}  \put(-290,30){$1$} \put(-275,30){$2$} \put(-260,30){$3$}
 \put(-163,60){$(h',e')$}  \put(-76,20){$(h,e)$}
  \put(-283,68){{\small $+$}} \put(-269,58){$-$} \put(-283,40){$+$} \put(-269,48){$-$} 
\put(-100,65){$+$} \put(-100,55){$+$} \put(-100,38){$-$} \put(-120,38){$-$}
\put(-110,15){$-$} \put(-90,15){$+$}
%\put(-40,-20){$-$} \put(-40,-5){$+$} \put(-105,-32){$-$} \put(-80,-32){$+$}
\put(-283,24){{\small $+$}} \put(-269,15){$-$}
    \caption{$(h',e')$ is an upper parallel to $(h,e)$}
    \label{parallel}
  \end{center}
\end{figure}

The following proposition is key property in this paper.

\begin{proposition}\label{anotherdef}
Two Gauss diagrams $D$ and $D'$ are $L_{n}$-equivalent if and only if there exists an embedded pure virtual braid $(h,e)$ of degree $n$ such that $D^{(h,e)}$ equals to $D'$ up to a sequence of Reidemeister moves. 
\end{proposition}

\begin{proof}
A necessary condition is obvious. 
To prove a sufficient condition, we will show the following two statements (1) and (2).

\begin{itemize}
\item[(1)] If $D_2$ is obtained from $D_1$ by RI (RII or RIII, respectively) and then an $L_{n}$-move $(h_1,e_1)$ ($(h_2,e_2)$ or $(h_3,e_3)$, respectively), then $D_2$ is obtained from $D_1$ by an $L_{n'}$-move $(h'_1,e'_1)$ ($(h'_2,e'_2)$ or $(h'_3,e'_3)$, respectively)  ($n' \geq n$), then RI (RII or RIII, respectively), and then the sequence of RII's.  
\item[(2)] If $D_2$ is obtained from $D_1$ by an $L_{n}$-move $(h,e)$ and then another $L_n$-move $(h',e')$, then  $D_2$ is obtained from $D_1$ by a sequence of RII's and then an $L_{n}$-move $(h'',e'')$, and then a sequence of RII's.
\end{itemize}

By (1), (2) and Proposition \ref{movedegree}, if $D$ and $D'$ are $L_{n}$-equivalent, there is an $L_n$-move and a sequence of Reidemeister moves such that $D'$ is obtained from $D$ by the $L_n$-move and then the sequence of Reidemeister moves, which proves this proposition. 
 
We show (1). We consider the case of RI.
In Figure \ref{$L_n$ and RI}, these Gauss diagrams are identical except in a local place of near to RI in this figure. 
Here $D_1$ is the upper left and $D_2$ the upper right in Figure \ref{$L_n$ and RI}.
By gray line we mean an embedded pure virtual braid, where we omit orientations and signs of their arrows. 
Given an embedded pure virtual braid $(h_1,e_1)$, we can move the ends of its arrows out the arrow derived from RI by a sequence of RII's (the right part in Figure \ref{$L_n$ and RI}). 
We then can consider arrows derived from $(h_1, e_1)$ and new arrows as a new embedded pure virtual braid for $D_1$ (the lower right in Figure \ref{$L_n$ and RI}), which has degree $\geq n$.
We denote it by $(h'_1,e'_1)$.
Then, $D_2$ is obtained from $D_1^{(h'_1,e'_1)}$ by the RI and then the RII's.
%$D_1^{(h'_1,e'_1)}$ equals to $D_2$ up to RI and a sequence of RII's. 
Moreover, similar considerations apply to the other RI.

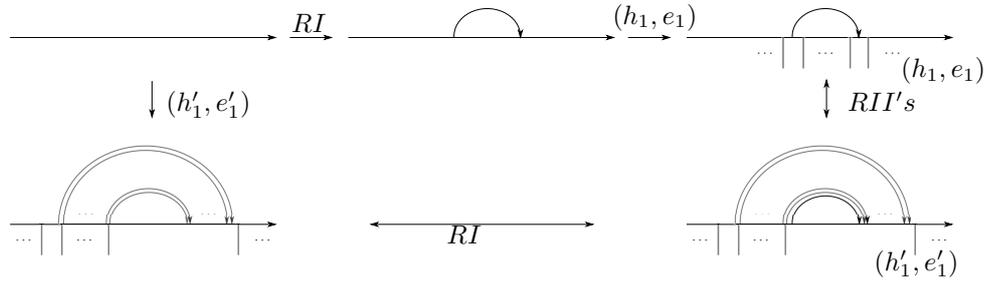
\begin{figure}[h]
  \begin{center}
% \raisebox{-22 pt}
  \includegraphics[height=3.3cm, scale=3]{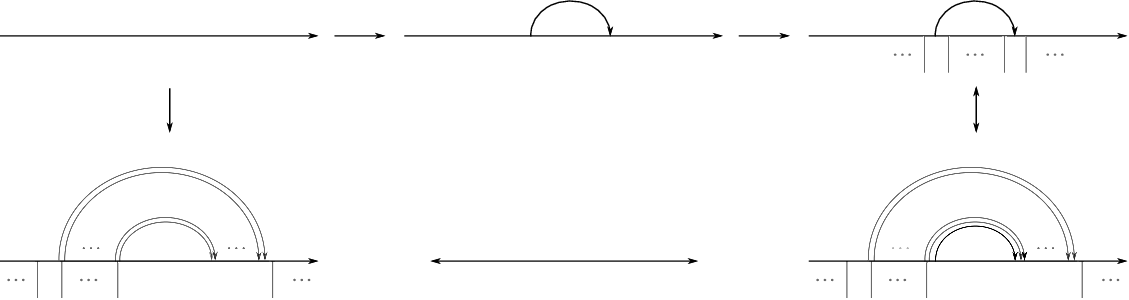}
 \put(-248,85){$RI$} \put(-128,88){$(h_1,e_1)$} \put(-295,55){$(h'_1,e'_1)$} \put(-20,68){$(h_1,e_1)$} \put(-40,55){$RII's$} \put(-190,5){$RI$} \put(-30,-5){$(h'_1,e'_1)$} 
    \caption{Change of an $L_n$-move and a first Reidemeister move}
    \label{$L_n$ and RI}
  \end{center}
\end{figure}

Similar way to RI, in the case of RII and RIII, we give embedded pure virtual braids $(h'_2,e'_2)$ and $(h'_3,e'_3)$ as in Figure \ref{$L_n$ and RII} and \ref{$L_n$ and RIII}, respectively, which are one of RII and RIII. 
Here, in Figure \ref{$L_n$ and RIII} for simplicity we draw only one strand embedded in each interval between endpoints of arrows derived from RIII.

\begin{figure}[h]
  \begin{center} %センタリングする
  \includegraphics[height=4.3cm, scale=4]{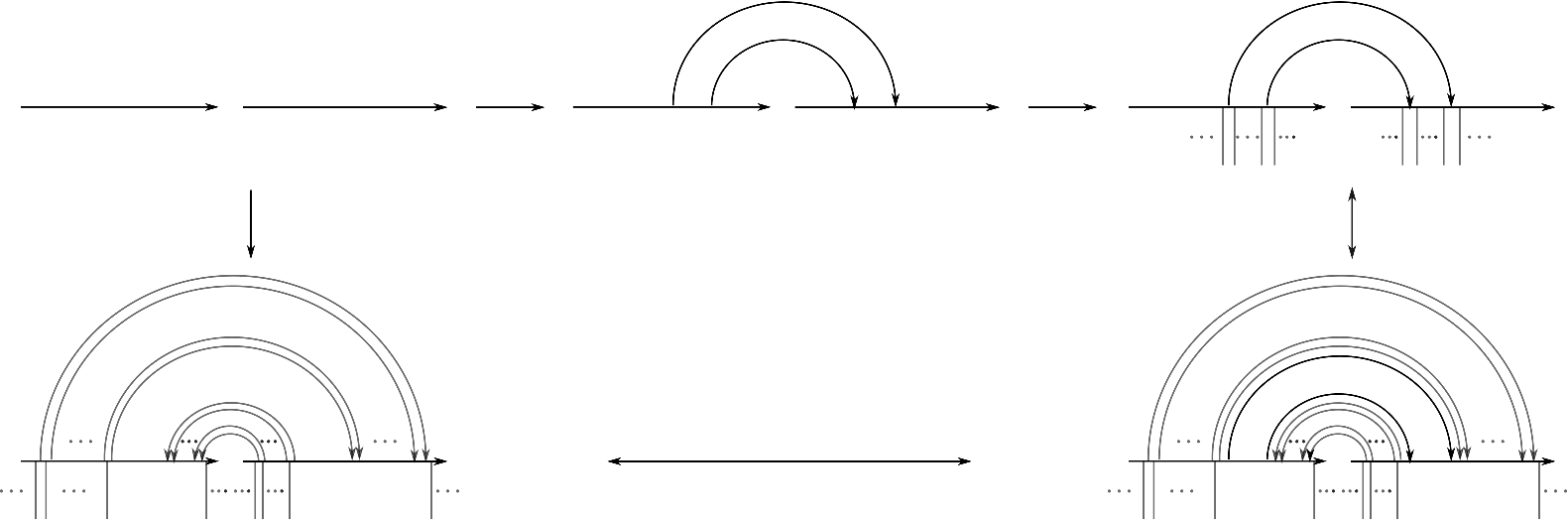} 
  \put(-260,105){$RII$} \put(-133,105){$(h_2,e_2)$} \put(-300,70){$(h'_2,e'_2)$} \put(-20,88){$(h_2,e_2)$} \put(-50,68){$RII's$} \put(-190,5){$RII$} \put(-30,-5){$(h'_2,e'_2)$} 
    \caption{Change of an $L_n$-move and a second Reidemeister move}
    \label{$L_n$ and RII}
 \end{center} 
\end{figure}

\begin{figure}[h]
  \begin{center} %センタリングする
  \includegraphics[height=4.8cm, scale=4]{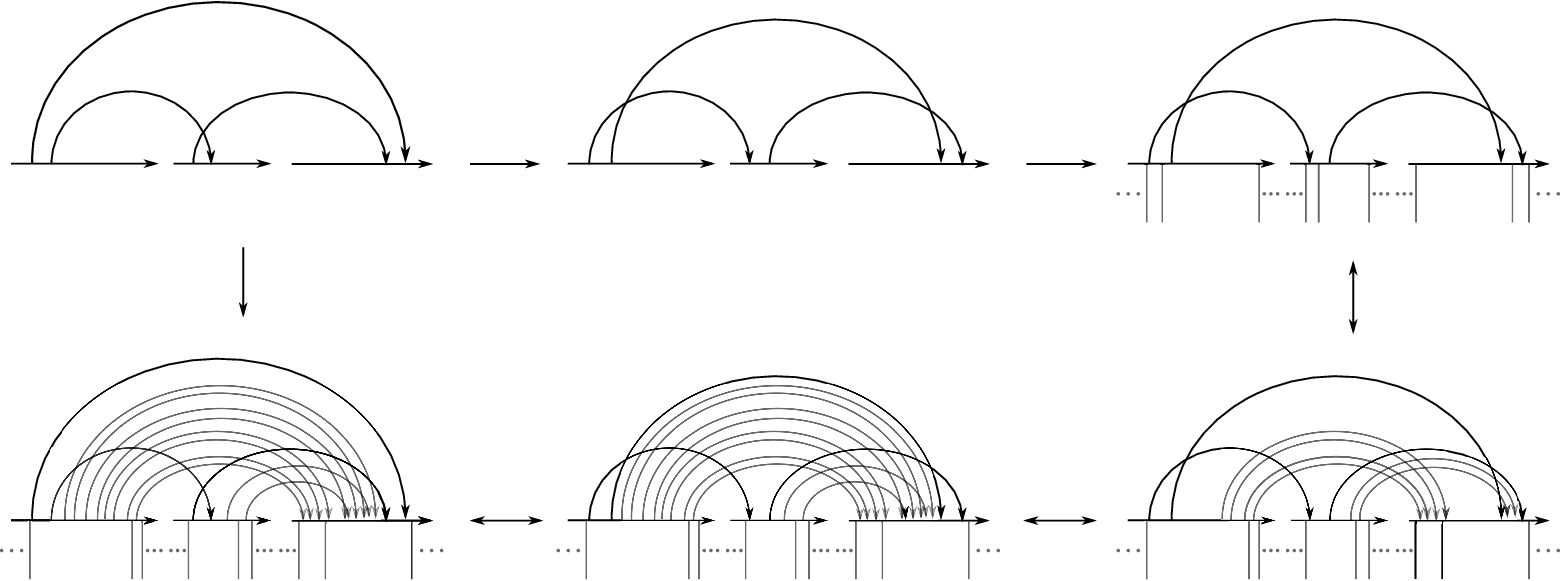}
  \put(-260,105){$RIII$} \put(-131,105){$(h_3,e_3)$} \put(-300,70){$(h'_3,e'_3)$} \put(-20,88){$(h_3,e_3)$} \put(-48,65){$RII's$} \put(-258,5){$RIII$} \put(-128,5){$RII's$} 
    \caption{Change of an $L_n$-move and a third Reidemeister move}
    \label{$L_n$ and RIII}
 \end{center} 
\end{figure}
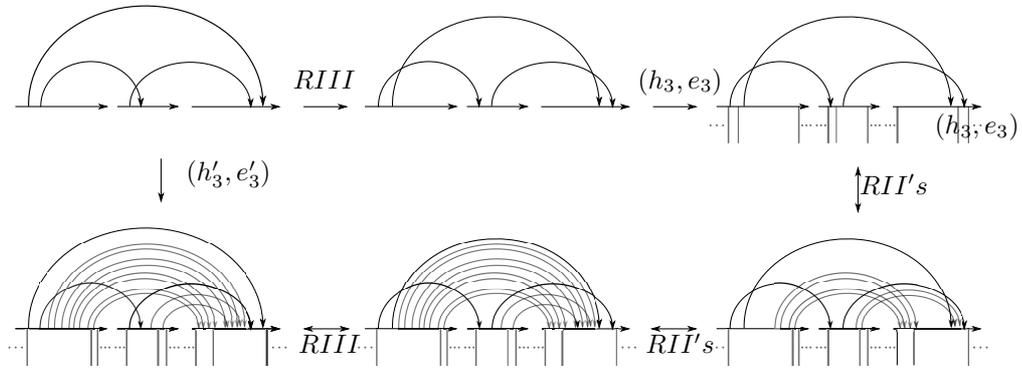

We show (2). 
We can transform $D_1$ to $D_1^{\{(h,e), (h^{-1},\bar{e})\}}$ by a sequence of RII's,
where $\bar{e}$ is upper parallel $e$ and $\bar{e}$ is disjoint from $e'$ for $D_1$.
Then, we can transform ${(D_1^{\{(h,e),(h^{-1},\bar{e})\}})}^{(h \otimes h',\tilde{e} \otimes e')}$ to $(D_1^{(h,e)})^{(h',e')}=D_2$ by a sequence of RII's, where $\tilde{e}$ is upper parallel to $\bar{e}$.
Here the degree of $(h \otimes h', \tilde{e} \otimes e')$ is $n$. 
We set $(h \otimes h',\tilde{e} \otimes e') $ by $(h'', e'')$ and the statement (2) holds. 
\end{proof}

%We call $H$ with embedding $e$  a {\it forest clasper} for $D$ of degree $n$.

\begin{remark}
It is obvious that Proposition \ref{anotherdef} is equivalent to the following statement.
There exists the union $H$ of disjoint embedded pure virtual braids of degree $n$ such that $D^{H}$ equals to $D'$ up to a sequence of the Reidemeister moves. 
\end{remark}

\begin{remark}
In \cite{MY}, Meilhan and Yasuhara introduced a family of local moves as an extension of $C_n$-equivalence to welded knots, which is a quotient of virtual knots.  
They also discuss the virtual knots in this paper. 
They proved that two virtual knots related by their moves are can not be distinguished by finite type invariants for each degree. 
It has still been open that it would hold the converse implication.
\end{remark}

\begin{lemma}\label{clasperproperty}
Let $n \geq 1$. Let $D$ be a Gauss diagram and $(h,e)$ an embedded pure virtual braid of degree $n$ for $D$.
Then for any Gauss diagram $D'$ which is equivalent to $D$ there is an embedded pure virtual braid $(h',e')$ of degree $n$ for $D'$ such that $D^{(h,e)}$ is equivalent to $D'^{(h',e')}$.
\end{lemma}

\begin{proof}
Since $D' \sim D$ and $D \overset{L_n}{\sim} D^{(h,e)}$, we have that $D' \overset{L_n}{\sim} D^{(h,e)}$.
It is from Proposition \ref{anotherdef} that there is an embedded pure virtual braid $(h',e')$ of degree $n$ for $D'$ such that $D^{(h,e)} {\sim} D'^{(h',e')}$.
\end{proof}

\begin{remark}
We can show Lemma \ref{clasperproperty} directly.   
If $D'$ is obtained from $D$ by Reidemeister move RI, RII or RIII, then given an embedded pure virtual braid $(h,e)$ for $D$ we can construct the pair $(h',e')$ such that $D^{(h,e)}$ is equivalent to $D'^{(h',e')}$ by similar method of Figure \ref{$L_n$ and RI}, \ref{$L_n$ and RII} and \ref{$L_n$ and RIII} in the proof of Proposition \ref{anotherdef}.
\end{remark}

The next proposition is well-known fact of group theory.
\begin{proposition}\label{2.25}%%2.25(S)
Let $G$ be a group. Let $x$ and $y$ be elements in the $n$-th and $n'$-th lower central subgroup of $G$, respectively. 
Then the commutator $[x, y]$ of $x$ and $y$ is in $(n+n')$-th lower central subgroup of $G$.
\end{proposition}

\begin{lemma}\label{slidedef}
Let $D$ be a Gauss diagram. 
Let $n_1$, $n_2 \geq 1$. Let $(h_1,e_1)$ and $(h_2,e_2)$ be disjoint  embedded pure virtual braids for $D$ of degree $n_1$ and  $n_2$, respectively. 
Let $s_i$ be the $i$-th strand of $h_1$ and $t_j$ the $j$-th one of $h_2$.
Suppose that $e_1(s_i)$ and $e_2(t_j)$ are on the same interval, and there is no endpoint of arrows and no embedding of another strands of embedded pure virtual braids on the intervals between embeddings $e_1(s_i)$ and $e_2(t_j)$.
Then, these embeddings may replace each other up to $L_{n_1+n_2}$-equivalence as in Figure \ref{slide}. 
Let $(h_1,e'_1)$ and $(h_2,e'_2)$ be embedded pure virtual braids of degree $n_1$ and $n_2$ obtained from $(h_1,e_1)$ and $(h_2,e_2)$ by replacing $e_1|_{s_i}$ and $e_2|_{t_j}$ as in Figure \ref{slide}. 
Then, there exists an embedded pure virtual braid $(h,e)$ for $D$ of degree $n_1+n_2$ such that $(h,e)$ is disjoint from both $(h_1,e_1)$ and $(h_2,e_2)$ and $(D^{\{(h_1,e_1), (h_2,e_2)\}})^{(h,e)}$ is equivalent to $D^{\{(h_1,e'_1),(h_2,e'_2)\}}$. 

\begin{figure}[h]
  \begin{center} %センタリングする
  \includegraphics[height=2.cm, scale=1.5]{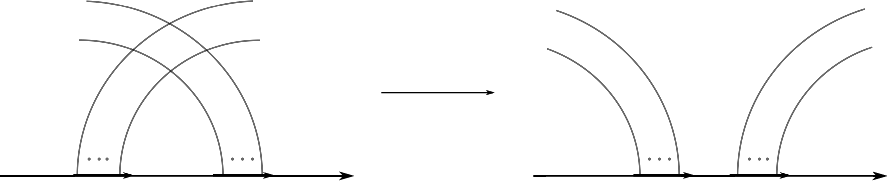} 
  \put(-285,-10){$D$} \put(-140,35){$(h_1,e'_1)$} \put(-285,35){$(h_1,e_1)$} \put(-0,35){$(h_2,e'_2)$} \put(-190,35){$(h_2,e_2)$} \put(-120,-10){$D$} 
  \put(-260,-20){$e_2(t_j)$} \put(-205,-20){$e_1(s_i)$} \put(-85,-20){$e'_1(s_i)$} \put(-45,-20){$e'_2(t_j)$} 
    \caption{A sliding}
    \label{slide}
 \end{center} 
\end{figure}
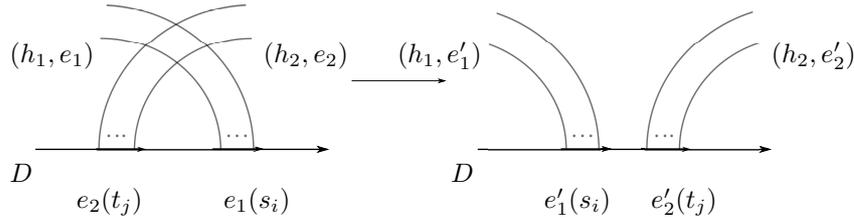

\end{lemma}

We call this transformation between two embedded pure virtual braids a {\it sliding}.

\begin{proof}
For given embedded pure virtual braids $(h_1,e_1)$ and $(h_2,e_2)$ with degree $n_1$ and $n_2$, respectively, we will construct an embedded pure virtual braid $(h,e)$ of degree $n_1+n_2$.
We consider the case both of the orientations of $e_1(s_i)$ and $e_2(t_j)$ are compatible with the orientation of an interval of $D$ and $e_1(s_i)$ is under than $e_2(t_j)$ (in the case of Figure \ref{slide}). 
In the other cases, we may construct it similarly.

First of all, we construct a pure virtual braid $h$.
Let $h_1 \in \Gamma_{n_1}(PV_{k_1})$, $h_1 \notin \Gamma_{n_1+1}(PV_{k_1})$, $h_2 \in \Gamma_{n_2}(PV_{k_2})$ and $h_2 \notin \Gamma_{n_2}(PV_{k_2+1})$.
Then, we define $\tilde{h}_1$ as an element of $\Gamma_{n_1}(PV_{k_1+k_2-1})$ obtained from $h_1$ by adding $j-1$ strands before the 1st strand of $h_1$ and $k_2-j$ strands after the $k_1$-th strand of $h_1$ and $\tilde{h}_2$ as an element of $\Gamma_{n_2}(PV_{k_1+k_2-1})$ obtained from $h_2$ by adding $i-1$ strands between the $(j-1)$-th and $j$-th strand of $h_2$, and 
$k_1-i$ strands between the $j$-th and $(j+1)$-th strand of $h_2$. 
Then the $i$-th strand of $h_1$ and $j$-th strand of $h_2$ are the same order $(i+j-1)$th in $\tilde{h}_1$ and $\tilde{h}_2$.
We define $h=[\tilde{h}_1, \tilde{h}_2]$, which is in $\Gamma_{n_1+n_2}(PV_{k_1+k_2-1})$ by Proposition \ref{2.25}. 

%\begin{figure}[h]
%  \begin{center} %センタリングする
%  \includegraphics[height=4.3cm, scale=4]{twobraid.eps} 
%\put(-170,90){$\tilde{h}_1=$} \put(-130,128){$j-1$} \put(-26,128){$k_2 -j$} \put(-102,120){$1$} \put(-70,120){$i$} \put(-37,120){$k_1$} 
%\put(-170,25){$\tilde{h}_2=$} \put(-138,55){$1$} \put(-96,55){$i-1$} \put(-58,55){$k_2-i$} \put(-122,55){$j-1$} \put(-70,55){$j$} \put(-30,55){$j+1$} \put(-8,55){$k_2$}
%%\caption{}
%%    \label{}
% \end{center} 
%\end{figure}

Secondly, we construct an embedding $e$. We define embeddings $\tilde{e}_1$ of $\tilde{h}_1$ and $\tilde{e}_2$ of $\tilde{h}_2$ as follows.
For each strand of $\tilde{h}_1$ (or $\tilde{h}_2$, respectively) derived from $h_1$ (or $h_2$, respectively),
 $\tilde{e}_1$ (or $\tilde{e}_2$, respectively) is the same as $e_1$  (or $e_2$, respectively), 
 and for each other strand of $\tilde{h}_1$ (or $\tilde{h}_2$, respectively), $\tilde{e}_1$ (or $\tilde{e}_2$, respectively) is under (or upper, respectively) parallel embedding to $e_2$ (or $e_1$, respectively). 
Then, $\tilde{e}_1$ is under parallel to $\tilde{e}_2$ and $D^{\{(h_1,e_1),(h_2,e_2)\}} \sim D^{\{(\tilde{h}_1,\tilde{e}_1),(\tilde{h}_2, \tilde{e}_2)\}}$.
Moreover, we define  an embedding $e$ of $h$ as the upper parallel embedding to $\tilde{e}_2$. 
%Then $h \in \Gamma_{n_1+n_2}(PV_{k_1+k_2-1})$ by Proposition \ref{2.25}.
Then, $D^{\{(\tilde{h}_1,\tilde{e}_1),(\tilde{h}_2, \tilde{e}_2)\}} \overset{L_n}{\sim} D^{\{(\tilde{h}_1,\tilde{e}_1),(\tilde{h}_2, \tilde{e}_2), (h,e)\}} \sim D^{(h_1 \cdot \tilde{h}_2 \cdot \tilde{h}_1,e)} \sim D^{(\tilde{h}_1 \cdot \tilde{h}_2,e)} \sim D^{(h_1, e'_1), (h_2, e'_2)}$. 
Similarly, we can construct $(h,e)$.
% that $D^{\{({h}_1,{e}_1),({h}_2, {e}_2)\}} \overset{L_n}{\sim} D^{(h_1, e'_1), (h_2, e'_2)}$ for other cases.
%, the product $\tilde{h}_1 \cdot \tilde{h}_2$ has a natural embedding $e$ induced by $e_1$ and $e_2$. 
%Let $h=[\tilde{h}_2, \tilde{h}_1]$. 
%Then $h \in \Gamma_{n_1+n_2}(PV_{k_1+k_2-1})$ by Proposition \ref{2.25}.
%Since $h \cdot  \tilde{h}_1 \cdot  \tilde{h_2} =  \tilde{h_2} \cdot \tilde{h_1}$, we may replace $s_i$ and $t_j$ each other and leave other embeddings up to $L_{n_1+n_2}$-equivalence.
%If only one of the orientations of $s_i$ and $t_j$ is compatible with that of the interval of D, to adjust  $s_i$ and $t_j$ the orientation  we set $h=[\bar{h}_2, \tilde{h}_1]$, where $\bar{h}_2$ is the mirror image of $\tilde{h}_2$ for a horizontal line and its embedding $e$ is induced by $e_1$ and $e_2$. 
\end{proof}

\begin{definition}
A Gauss diagram $D$ is  {\it $L_n$-trivial} if $D$ is $L_{n}$-equivalent to the trivial Gauss diagram. 
\end{definition}

\begin{proposition}\label{5.8}%%5.8(H)
Let $n, n' \geq 1$.
Let $D$ be an $L_n$-trivial Gauss diagram and $D'$ be an $L_{n'}$-trivial one.
Then the Gauss diagram $D \cdot D'$ is $L_{n+n'}$-equivalent to $D' \cdot D$.
\end{proposition}

\begin{proof}
By assumption and Proposition \ref{anotherdef}, there are two embedded pure virtual braids $(h,e)$ and $(h',e')$ of degree $n$ and $n'$ such that $D_0^{(h,e)} \sim D$ and $D_0^{(h',e')} \sim D'$, respectively.  
Then by Lemma \ref{slidedef} we have $D \cdot D' \sim D_0^{(h,e)} \cdot D_0^{(h',e')} \overset{L_{n+n'}}{\sim}  D_0^{(h',e')}  \cdot D_0^{(h,e)} \sim D' \cdot D$.
\end{proof}

\begin{proposition}\label{1.39}%%1.39(S)
For any $L_n$-trivial Gauss diagram $D$, there is an $L_n$-trivial Gauss diagram $D'$ such that both $D \cdot D'$ and $D' \cdot D$ are $L_{2n}$-trivial. 
%For any $L_n$-trivial long virtual knot $K$, there is $L_n$-trivial long virtual knot $K'$ such that both $K \cdot K'$ and $K' \cdot K$ are $L_{2n}$-trivial. 
\end{proposition}

\begin{proof}
By assumption and Proposition \ref{anotherdef}, there is an embedded pure virtual braid $(h,e)$ of degree $n$ such that $D_0^{(h,e)} \sim D$.
We define $D'=D_0^{(h^{-1},e)}$. 
Then by Lemma \ref{slidedef} we have $D \cdot D' \sim D_0^{(h,e)} \cdot D_0^{(h^{-1},e)} \overset{L_{2n}}{\sim} D_0^{(h \cdot h^{-1},e)} \sim D_0$ and similarly we have $D' \cdot D \overset{L_{2n}}{\sim} D_0$. 
\end{proof}

\begin{notation} 
The set $\mathcal{VSL}(k)$ of equivalence classes of Gauss diagrams on $k$ strands has a monoid structure under the composition for any positive integer $k$. 
For $n \geq 1$, let $\mathcal{VSL}_n(k)$ denote the submonoid of $\mathcal{VSL}(k)$ consisting of the equivalence classes of Gauss diagrams on $k$ strands which are $L_n$-trivial.
There is a descending filtration of monoids 
\[ \mathcal{VSL}(k) = \mathcal{VSL}_1(k) \supset \mathcal{VSL}_2(k)  \supset \mathcal{VSL}_3(k)  \supset \cdots. \]

For $l \geq n$,  $\mathcal{VSL}_n(k) /  {L_l}$ denotes the quotient of $\mathcal{VSL}_n(k)$ by $L_l$-equivalence.
It is easy to see that the monoid structure on $\mathcal{VSL}_n(k)$ induces that of $\mathcal{VSL}_n(k) / {L_l}$.
There is a filtration on $\mathcal{VSL}_n(k) / {L_l}$ of finite length 
\[ \mathcal{VSL}/ {L_l} = \mathcal{VSL}_1(k)/ {L_l} \supset \mathcal{VSL}_2(k)/  {L_l}  \supset \mathcal{VSL}_3(k)/ {L_l}  \supset \cdots   \supset \mathcal{VSL}_l(k)/  {L_l}=\{1\}. \]
\end{notation}

\begin{lemma}\label{abelian}
For $n, k\geq 1$, the monoid $\mathcal{VSL}_n(k)/ {L_{l}}$ $(1 \leq l \leq 2n)$ is an abelian group.
\end{lemma}

\begin{proof} 
By Proposition \ref{1.39}, for any $L \in \mathcal{VSL}_n(k)$ there exists $L' \in \mathcal{VSL}_n(k)$ such that both $L \cdot L'$ and $L \cdot L'$ are trivial up to $L_{2n}$-equivalence, 
and therefore the monoid $\mathcal{VSL}_n(k)/ {L_{2n}}$ is a group.
By Proposition \ref{5.8}, for any $L_1, L_2 \in \mathcal{VSL}_n(k)$, $L_1 \cdot L_2 $ is $L_2 \cdot L_1$ up to  $L_{2n}$-equivalence, and the group $\mathcal{VSL}_n(k)/  {L_{2n}}$ is abelian. 
It follows from Proposition \ref{movedegree} that it holds for $1 \leq l < 2n$. 
\end{proof}

\begin{proposition}\label{5.4}%%5.4(H)
The monoid $\mathcal{VSL}_n(k)/ {L_{l}}$ is a nilpotent group for any positive integers $k$, $n$ and $l$. 
\end{proposition}

\begin{proof}
We fix $l$ and prove it by induction on $n$.
By Proposition \ref{movedegree}, it is obvious for $n \geq l$.
Assume that $\mathcal{VSL}_{n+1}(k)/ {L_{l}}$ is a group for some $n$ with $1 \leq n \leq l$. We then have a short exact sequence of monoids:
\[1 \rightarrow \mathcal{VSL}_{n+1}(k)/ {L_{l}} \rightarrow \mathcal{VSL}_{n}(k)/ {L_{l}} \rightarrow \mathcal{VSL}_{n}(k)/ {L_{n+1}}   \rightarrow 1. \]
Here, $\mathcal{VSL}_{n+1}(k)/ {L_{l}}$ and $\mathcal{VSL}_{n}(k)/ {L_{n+1}} $ are groups by the assumption of induction and Lemma \ref{abelian}.
Therefore $\mathcal{VSL}_{n}(k)/ {L_{l}}$ is also a group. 
Moreover, it is from Proposition \ref{5.8} that $[\mathcal{VSL}_n(k)/ {L_{l}}, \mathcal{VSL}_{n'}(k)/  {L_{l}}] \subset  \mathcal{VSL}_{n+n'}(k)/ {L_{l}}$ for any $n, n', l \geq 1$
Therefore, $\mathcal{VSL}_n(k)/ {L_{l}}$ is nilpotent. 
\end{proof}

%%%%%%%%%%%%%%%%%%%%%%%%%%%%%%%%%%%%%%%%%%%%%%%%%%%%%%%%%%%%%%%%%%%%%%%%%%%%%%%%%%%%%%%%%%%%%%%%%%%%%%%%%%%%%%%%%%%%%%%%%%%%%%%%%%
\section{$L_n$-equivalence and $n$-equivalence}

In this section, we introduce the $n$-equivalence for Gauss diagrams on several strands defined by Goussarov-Polyak-Viro \cite{GPV} and prove that the $L_n$-equivalence coincides with $(n-1)$-equivalence on virtual string links. 

\begin{definition}\cite{GPV}\label{trivial}
Let $n \geq 0$. A Gauss diagram $D$ on several strands is said to be {\it $n$-trivial} (with respect to $A_1, A_2, \cdots , A_{n+1}$) if the Gauss diagram satisfies the following condition. 
There exist $n+1$ non-empty disjoint subsets $A_1, A_2, \cdots , A_{n+1}$ of the set of arrows of $D$ such that 
for any non-empty subfamily $S$ of the set $\{A_1, A_2, \cdots , A_{n+1}\}$ the Gauss diagram obtained from $D$ by  removing all arrows which belongs to $S$ is trivial up to a sequence of second Reidemeister moves.

A Gauss diagram $D_2$ is related to a Gauss diagram $D_1$  by {\it $n$-variation} if $D_2$ is obtained from $D_1$ by attaching an $(n-1)$-trivial Gauss diagram on several strands to segments of $D_1$ without endpoints of any arrow.  
Two Gauss diagrams are said to be  {\it $n$-equivalent} if they are related by a sequence of $(n+1)$-variations and Reidemeister moves.
\end{definition}

In order to prove the next theorem (Theorem~\ref{L_nandn-equ}), we prepare the following definition and two  lemmas.

We define a {\it weight} for an embedded pure virtual braid $(h,e)$, which is a finite subset of $\mathbb{N}$, and denote it by $w(h,e)$.
Let $H$ be a finite set of embedded pure virtual braids with weights. Let $I$ be a finite subset of $\mathbb{N}$. Then $H(I)$ denote the subset of $H$ each element of which has a subset of $I$ as a weight,  
and $H_I$ denote the subset of $H(I)$ each element of which has $I$ as a weight.

We define the weight of new embedded pure virtual braids obtained by sliding in Lemma \ref{slidedef} as follows.
When we slide two pairs $(h_1,e_1)$ and $(h_2,e_2)$ with weights, we define the weight of the deformed pairs $(h_1,e'_1)$, $(h_2,e'_2)$ and the new embedded pure virtual braid $(h,e)$ as $w(h_1,e_1)$, $w(h_2,e_2)$ and the union of these two weight, respectively.    
Then it is easy to see that $|w(h,e)|\leq  |w(h_1,e_1)| + |w(h_2,e_2)|$.

%\begin{lemma}\label{slidingweignt}
%Let $(h_1,e_1)$, $(h_2,e_2)$, $(h_1,e'_1)$, $(h_2,e'_2)$ and $(h,e)$ be embedded pure virtual braids for a Gauss diagram $D$ on several strands in Lemma \ref{slidedef}.
%Suppose that $(h_1,e_1)$ and $(h_2,e_2)$ have weights.
%Then, by the definition of the weight of the deformed and new embedded pure virtual braids by sliding, $D^{\{(h_1,e'_1), (h_2,e'_2)\}(I)} \sim D^{\{(h_1,e_1), (h_2,e_2),  (h,e) \}(I)}$ for any subset $I$ of $\mathbb{N}$. 
%\end{lemma}

%\begin{proof}
%By the definition of the weight of the deformed and new embedded pure virtual braids by sliding, it is obvious.
%\end{proof}

\begin{lemma}\label{slidingweignt}
Let $H$ and $H'$ be sets of finite embedded pure virtual braids with weights for a Gauss diagram $D$ on several strands such that they are related by sliding in Lemma \ref{slidedef} (this is, $D^H \sim D^{H'}$).
Then $D^{H(I)}$ is equivalent to $D^{H'(I)}$ for any subset $I$ of $\mathbb{N}$. 
\end{lemma}

\begin{proof}
Let $(h_1,e_1)$, $(h_2,e_2)$, $(h_1,e'_1)$, $(h_2,e'_2)$ and $(h,e)$ be embedded pure virtual braids for a Gauss diagram $D$ on several strands in Lemma \ref{slidedef}.
Suppose that $(h_1,e_1)$ and $(h_2,e_2)$ have weights.
Then, by the definition of their weights by sliding, $D^{\{(h_1,e'_1), (h_2,e'_2)\}(I)} \sim D^{\{(h_1,e_1), (h_2,e_2),  (h,e) \}(I)}$ for any subset $I$ of $\mathbb{N}$. 
\end{proof}

\begin{lemma}\label{weigntdegree}
(1) Let $D$ and $D'$ be Gauss diagrams on several strands related by RI (RII or RIII, respectively).  
Let $H=\{(h_i,e_i)\}_i$ be a set of finite embedded pure virtual braids with weights for $D$.
We then have a set $H'=\{(h'_i,e'_i)\}_i$ of finite embedded pure virtual braids with weights for $D'$ such that $D^{H(I)}$ is equivalent to ${D'}^{H'(I)}$ up to a sequence of RI (RII or RIII, respectively) and RII's for any subset $I$ of $\mathbb{N}$ and the degrees of $(h_i,e_i)$ and $(h'_i,e'_i)$ are the same for any $i$.

In particular, (2) if there exist disjoint sets $A_1, A_2, \cdots , A_{n+1}$ of arrows such that $H=A_1 \cup A_2 \cup \cdots \cup A_{n+1}$ as a set of arrows and $D^H$ has $n$-triviality with respect to $A_1, A_2, \cdots , A_{n+1}$, then there exist disjoint sets $A'_1, A'_2, \cdots , A'_{n+1}$ of arrows such that $H'=A'_1 \cup A'_2 \cup \cdots \cup A'_{n+1}$ as a set of arrows and ${D'}^{H'}$ has $n$-triviality with respect to $A'_1, A'_2, \cdots , A'_{n+1}$.  
\end{lemma}

\begin{proof}
(1) By using the method of (1) in the proof of Proposition \ref{anotherdef}, we can construct $H'$ from $H$.
Here the relation between $(h_i,e_i)$ and $(h'_i,e'_i)$ in  the proof of Lemma \ref{anotherdef} corresponds to that in Proposition \ref{weigntdegree}.
In the proof of Proposition \ref{anotherdef}, the degree of $(h'_i,e'_i)$ and $(h_i,e_i)$ are the same for any $i$ and we define the weight of $(h'_i,e'_i)$ by the weight of $(h_i,e_i)$.

(2) In particular, we define $H_{\{i\}}=A_i$ and $H(\{1,2, \cdots n+1 \})=H$.
% and $H'(\{i\})=A'_i$ and $H'(\{1,2, \cdots n+1 \})=H'$. 
We then have that if $D^H$ has $n$-triviality with respect to $A_1, A_2, \cdots , A_{n+1}$, then ${D'}^{H'}$ has also $n$-triviality with respect to $A'_1, A'_2, \cdots , A'_{n+1}$ where $H'_{\{i\}}=A'_i$ and $H'(\{1,2, \cdots n+1 \})=H'$.
\end{proof}

\begin{theorem}\label{L_nandn-equ}%%
For any $n \geq 1$, the $L_{n}$-equivalence and ${(n-1)}$-equivalence on virtual string links are equal.
\end{theorem}

\begin{proof}
A pure virtual braid $[a_1,[a_2, \cdots [a_{n-1}, a_{n}] \cdots ]]$ has $(n-1)$-triviality with respect to $A_1, A_2, \cdots , A_{n}$ such that $A_i$ is the set of all $a_i$ and $a_i^{-1}$, where $a_i$ is a generator of a group.
Therefore any element of the $n$-th lower central series of the pure virtual braid group has $(n-1)$-triviality.
Therefore we have that if two Gauss diagrams are $L_{n}$-equivalent, then they are ${(n-1)}$-equivalent.
Therefore it suffices to prove that if a Gauss diagram $D'$ is related to $D$ by an $n$-variation then they are $L_{n}$-equivalent.

Let $D_t$ be an $(n-1)$-trivial Gauss diagram on several strands with respect to  $A_1, A_2, \cdots , A_{n}$ such that $D'$ is obtained from $D$ by attaching $D_t$.
Let $D_A$ be the Gauss diagram obtained from $D_t$ by removing all arrows in $A_1\cup\cdots\cup A_n$.
By the property of $(n-1)$-triviality,  $D_A$ equsls to the trivial Gauss diagram $D_0$ up to a sequence of RII's, where $D_0$ has the same number of intervals as $D_t$.
We can consider each arrow in $A_1\cup\cdots\cup A_n$ as an embedded pure virtual braid of degree 1 for $D_A$ and denote their set by $H$, that is $H=A_1\cup\cdots\cup A_n$, and $D_t=D_A^H$.    
It follows from Lemma \ref{weigntdegree}(2) that we have a set $H'=A'_1\cup\cdots\cup A'_n$ of embedded pure virtual braids for $D_0$ with degree 1 such that $H'$ has $(n-1)$-triviality with respect to $A'_1, A'_2, \cdots , A'_{n}$ and $D_t$ equals to $D_0^{H'}$ up to a sequence of RII's. 
%, where $D_0$ is the trivial Gauss diagram with the same number of strands as $D_t$.  
Therefore $D'$ is equivalent to a Gauss diagram $D''$ obtained from $D$ by attaching $D_0^{H'}$ to the same segments of $D$ as $D_t$. 
Here, $D''$ can be regarded as $D^{H'}$ by considering $H'$ as the set of embedded pure virtual braids for $D$.  
Then $D \overset{L_n}{\sim} D^{H'}$ shows $D \overset{L_n}{\sim} D'$.
Therefore we show the following claim, which proves the theorem.

\begin{claim}
Let $D$ be a Gauss diagram and $H'=A'_1\cup\cdots\cup A'_n$ the set of embedded pure virtual braids of degree 1 for $D$ with $(n-1)$-triviality with respect to $A'_1, A'_2, \cdots , A'_{n}$. 
Then $D$ is $L_n$-equivalent to $D^{H'}$.
\end{claim}

Let us first prove the case that $D$ is the trivial Gauss diagram $D_0$.
We consider $H'$ as a set of embedded pure virtual braids with weights each element of which assigns $i$ as a weight if it is in $A_i$.
%By the property of $(n-1)$-triviality if $I$ is a proper subset of $N=\{1,2, \cdots ,n\}$, then $D_0^{H'(I)}$ is equivalent to $D_0$ up to a sequence of RII's.
We show the following statement, which proves the claim.

(A) 
For any $s \in N=\{1,2, \cdots ,n\}$, there exists a set $G_s$ of finite embedded pure virtual braids for $D_0$ with weights such that $s \leq |w(h,e)| \leq$ deg$(h,e)$ for each element $(h,e)$ of $G_s$, where $|\cdot|$ means the number of a set, and $D_0^ {G_s(I)} \sim D_0^{H'(I)}$ for every subset $I$ of $N$.

%To prove the statement (A), we prove the following statement depending on a positive integer $s$.

We prove it by induction on $s$ for $s=1, 2, \cdots, n$.
For $s=1$, we can set $G_1=H'$.
Under the assumption of the claim for $H'$, assuming the statement (A) to hold for $s(<n)$, we will prove it for $s+1$.
Let $G_s$ be a set of finite embedded pure virtual braids for $D_0$ satisfying (A) for $s$, that is, $D_0^{G_s(I)} \sim D_0^{H'(I)}$ for any subset $I$ of $N$ and $s \leq |w(h,e)| \leq \text{deg}(h,e)$ for each $(h,e) \in G_s$. 
We take a subset $J$ of $N$ such that $|J|=s$.
We then shift all embedded pure virtual braid in ${(G_s)}_J$ to the ahead of the intervals with fixing embedded pure virtual braids in $G_s \setminus {(G_s)}_J$ by their sliding (Lemma \ref{slidedef}) 
until all endpoints of all element in ${(G_s)}_J$ are completely to their ahead in $G_s \setminus {(G_s)}_J$, where we do not slide between elements in ${(G_s)}_J$. 
We denote the set of the obtained embedded pure virtual braids for $D_0$ by $G'$.
Then, it follows from Lemma \ref{slidingweignt} that  $D_0^{G_s(I)} \sim D_0^{G'(I)}$ for every $I \subset N$.
%Here we define the weight of new tree claspers obtained by sliding in Lemma \ref{slidedef} as follows.
%If two claspers $(h_1,e_1)$ and $(h_2,e_2)$ have weights, then the new tree clasper $(h,e)$ has the union of these two weight. 
Moreover, $s+1 \leq |w(h,e)|$ for each new embedded pure virtual braid $(h,e)$ and if $(h,e)$ is derived from sliding between $(h_1,e_1)$ and $(h_2,e_2)$ then $|w(h,e)|  \leq |w(h_1,e_1)| + |w(h_2,e_2)|$.
On the other hand, by Lemma \ref{slidedef}, $\text{deg}(h_1,e_1) + \text{deg}(h_2,e_2) = \text{deg}(h,e)$.
Therefore it follows from the assumption of (A), for any new embedded pure virtual braid $(h,e) \in G'$
%\begin{align*} 
%s+1 &\leq |w(h,e)| \\
%&\leq |w(h_1,e_1)| + |w(h_2,e_2)| \\
%&\leq \text{deg}(h_1,e_1) + \text{deg}(h_2,e_2) = \text{deg}(h,e).
%\end{align*}
$$s+1 \leq |w(h,e)| \leq |w(h_1,e_1)| + |w(h_2,e_2)| \leq \text{deg}(h_1,e_1) + \text{deg}(h_2,e_2) = \text{deg}(h,e).$$

%We consider a subset $K$ of $N$.
Next we show that $D_0^{G'(I)} \sim (D_0^{G'_J})^{(G' \setminus G'_J)(I)}$ for every $I \subset N$.
If $I \supset J$, then it is clear that  $D_0^{G'(I)} = (D_0^{G_{J}'})^{(G' \setminus G'_J)(I)}$.
If not, then the new embedded pure virtual braids are not contained in $G'(I)$ and hence $(D_0^{G'_J})^{(G' \setminus G'_J)(I)}=D_0^{G'(I)} \cdot D_0^{G'_J}$.
By the property of $(n-1)$-triviality in the assumption of the claim, $D_0^{H'(J)} \sim D_0$, which corresponds to $(n-1)$-triviality with respect to a subfamily $S$ (in Definition \ref{trivial}) of $\{A'_1, A'_2, \cdots , A'_n\}$ consisting of $n-s$ elements.
By the assumption of (A), $D_0^{G'(J)} = D_0^{G'_J}$.
Therefore $D_0^{G'_J} \sim D_0^{G'(J)} \sim D_0^{G_s(J)} \sim D_0^{H'(J)} \sim D_0$ and so $D_0^{G'(I)} \sim (D_0^{G'_J})^{(G' \setminus G'_J)(I)}$.
%$D_0^{H'(J)} \sim D_0^{G_s(J)} \sim 

Now we regard $G' \setminus G'_J$ as a set of embedded pure virtual braids for $D_0^{G'_J}$. 
It follows from Lemma \ref{weigntdegree}(1) that there exists a set ${G''}$ of finite embedded pure virtual braids for $D_0$ such that $(D_0^{G'_J})^{(G' \setminus G'_J)(I)} \sim D_0^{G''(I)}$ for any $I \subset N$.
Note that $G''_J = \emptyset$. 
Thus all embedded pure virtual braids with weights $J$ are eliminated.
Repeating this procedure for all other $J$ such that $|J|=s$, we have a set $G_{s+1}$ of finite embedded pure virtual braids for $D_0$ such that $D_0^{G_{s+1}(I)} \sim D_0^{H'(I)}$ for any $I \subset N$ and $s+1 \leq |w(h,e)| \leq  d(h,e)$ for any $(h,e)$ in $G_{s+1}$, which is the required set satisfying (A) for $s+1$.
This proves the claim for the case that $D = D_0$.

%$(h,e)$'s with $s+1 \leq |w(h,e)| \leq  d(h,e)$ for a Gauss diagram which is equivalent to .
%By Lemma \ref{clasperproperty}, we obtain a set of tree claspers preserving above condition for $D_0$, which is the required set satisfying (A) for $s+1$.
%This proves the claim for the case that $D = D_0$.

Next we consider the case that $D \sim D_0$.
For a given $H'$ for $D$ in the claim, by Lemma \ref{weigntdegree}(2) there exists a set $H''$ of finite embedded pure virtual braids for $D_0$ with $(n-1)$-triviality such that $D^{H'} \sim D_0^{H''}$. 
Then, by the first case of the claim, $D_0^{H''} \overset{L_n}{\sim} D_0$.
Therefore $D^{H'} \overset{L_n}{\sim} D$ and 
%there exists another set $H'$ of claspers of degree $\geq n$ such that $D^{H'} \sim D_0^{H''}$. 
%Moreover, by Lemma \ref{clasperproperty}, there exists another set $H''$ of claspers of degree $\geq n$ such that $D^{H''} \sim D_0^{H'''}$. 
it proves the claim for the case that $D \sim D_0$.

Finally we prove the claim for the case that $D$ is not equivalent to $D_0$.
Since the set of $L_n$-equivalence classes has a group structure (Proposition \ref{5.4}), there is an inverse ${D}^{-1}$ of $D$ up to $L_n$-equivalence for any $n \geq 1$.
Now we show $D \cdot {D}^{-1} \overset{L_n}{\sim} D^{H'} \cdot D^{-1}$, which implies $D  \overset{L_n}{\sim} D \cdot D^{-1} \cdot D  \overset{L_n}{\sim} D^{H'} \cdot D^{-1} \cdot D \overset{L_n}{\sim} D^{H'}$. 
It follows from  $D \cdot {D}^{-1} \overset{L_n}{\sim} D_0$ and Proposition \ref{anotherdef} that there exists an embedded pure virtual braid $(h,e)$ of degree $n$ such that $(D \cdot {D}^{-1})^{(h,e)} \sim D_0$, where $(h,e)$ and $H'$ are disjoint.
Moreover it follows from the case $D \sim D_0$ of the claim that 
$(D^{H'} \cdot D^{-1})^{(h,e)} = ((D \cdot D^{-1})^{(h,e)})^{H'} \overset{L_n}{\sim} (D \cdot D^{-1})^{(h,e)}$.
Therefore $D^{H'} \cdot D^{-1} \overset{L_n}{\sim} (D^{H'} \cdot D^{-1})^{(h,e)} \overset{L_n}{\sim} (D \cdot D^{-1})^{(h,e)} \overset{L_n}{\sim} D \cdot D^{-1}$ and the proof of claim is proved.
\end{proof}

\begin{remark}
Even though we change``second Reidemeister moves" into``Reidemeister moves" in the definition of the $n$-trivial in Definition \ref{trivial}, 
we can show Theorem \ref{L_nandn-equ} similarly. 
Therefore it is concluded that these two $n$-equivalences coincide. 
\end{remark}

%\begin{figure}[ht]
%$$\raisebox{-28 pt}{\begin{overpic}[bb=0 0 116 126, height=80
%pt]{SepLeaf01.eps}
%\put(-12,60){$f$} \put(27,48){$a$} \put(28,-11){ $T$}
%\end{overpic}}$$
%  \hspace{0.4cm} \mbox{\large$\longrightarrow$} \hspace{0.6cm}
%\caption{Separating of leaf.} \label{separate}
%\end{figure}

%%%%%%%%%%%%%%%%%%%%%%%%%%%%%%%%%%%%%%%%%%%%%%%%%%%%%%%%%%%%%%%%%%%%%%%%%%%%%%%%%%%%%%%%%%%%%%%%%%%%%%%%%%%%%%%%%%%%%%%%%%%%%%%%%%
\section{$L_n$-equivalence and $V_n$-equivalence}

Goussarov-Polyak-Viro \cite{GPV} mentioned that the value of a finite type invariant of degree less than or equal to $n$ depends only on the $n$-equivalence classes.
Therefore it follows from Theorem \ref{L_nandn-equ} that $L_{n+1}$-equivalence implies $V_n$-equivalence, indirectly. 
In this section, we give this relation directly, by redefining the two-sided ideal $J_n(k)$ of the monoid ring $\mathbb{Z}\mathcal{VSL}(k)$ by using embedded pure virtual braids.

\begin{definition}%%6.5(H)
Let $l \geq 1$.
Let $H=\{ (h_1,e_1), (h_2,e_2), \dots ,(h_l,e_l) \}$ be a set of $l$ disjoint embedded pure virtual braids for $D$ on $k$ strands.
%A {\it scheme} of size $l$, $H=\{ (h_1,e_1), (h_2,e_2), \dots ,(h_l,e_l) \}$, for a Gauss diagram $D$ on several strands is the set of disjoint claspers for $D$ on $k$ strands.
Denote an element $[[D,H]]$ of $\mathbb{Z}\mathcal{VSL}(k)$ by 
$$ [[D, H]] = \sum_{G \subset H} (-1)^{l - \#G} [D^{G}],$$
where $G$ runs over all $2^l$ subsets of $H$.
We define $[[D,\emptyset]] = [D]$.
The {\it degree} of $H$ is defined by the sum of the degree of its all elements, denoted by deg($H$).
\end{definition}

It is easy to see that 
\begin{align*}
& [[D, \{(h_1,e_1), (h_2,e_2), \dots ,(h_l,e_l) \}]] \\ & \hspace{.3cm}=[[D^{(h_1,e_1)}, \{ (h_2,e_2), \dots ,(h_l,e_l)\}]] - [[D, \{(h_2,e_2), \dots ,(h_l,e_l)\}]]. \\
\end{align*}

\begin{lemma}\label{clasperproperty2}
Let $D$ be a Gauss diagram on $k$ strands and $H$ a set of $l$ disjoint embedded pure virtual braids for $D$ of degree $n$.
Then for any Gauss diagram $D'$ which is equivalent to $D$ there is a set $H'$ of $l$ disjoint embedded pure virtual braids for $D'$ of degree $n$ such that $[[D, H]]$ is equal to $[[D', H']]$ in $\mathbb{Z}\mathcal{VSL}(k)$.
\end{lemma}

\begin{proof}
If $D'$ is obtained from $D$ by RI, RII or RIII, we can construct a set $H'$ of $l$ disjoint embedded pure virtual braids for $D'$ of degree $n$ such that $D^H \sim D'^{H'}$ by similar method of Figure \ref{$L_n$ and RI}, \ref{$L_n$ and RII} and \ref{$L_n$ and RIII} in the proof of Proposition \ref{anotherdef}.
From the construction of $H'$, for each $G \subset H$, there is the corresponding $G' \subset H'$ such that $D^G \sim D'^{G'}$.
Therefore $[[D, H]] = [[D', H']]$ in $\mathbb{Z}\mathcal{VSL}(k)$.  
\end{proof}

\begin{definition}
Let $n$, $l$ be integers with $1 \leq l \leq n$. Let $J_{n,l}(k)$ denote the two-sided ideal of $\mathbb{Z}\mathcal{VSL}(k)$ generated by the elements $[[D, H]]$ under the composition, 
where $D$ is any Gauss diagram on $k$ strands and $H$ is any set of disjoint $l$ embedded pure virtual braids for $D$ of degree $n$. 
\end{definition}

\begin{remark}\label{ringiso}%%
The natural homomorphism $i : \mathbb{Z}\mathcal{VSL}(k)  \rightarrow  \mathbb{Z}(\mathcal{VSL}(k)/ {L_n})$ induces the ring isomorphism $\mathbb{Z}\mathcal{VSL}(k)/ J_{n,1}(k) \cong  \mathbb{Z}(\mathcal{VSL}(k)/{L_n})$.
\end{remark}

\begin{lemma}\label{Thm6.7}%%6.7(H)
Let $D$ be a Gauss diagram on $k$ strands. We then have the following properties. \\
(1) For any positive integer $n$, $J_{n, n}(k) =J_n(k)$ \\
(2) For any positive integers $n$, $l$, $l'$ with $l  \leq l' \leq n$, $J_{n, l}(k) \subset J_{n,l'}(k)$ \\
%(3) For any positive integers $n$, $n'$ $l$ with $l  \leq n \leq n'$, $J_{n', l} \subset J_{n,l}$
\end{lemma}

\begin{proof}
(1) We show that 
$$\varphi([\includegraphics[height=.4cm, scale=.8]{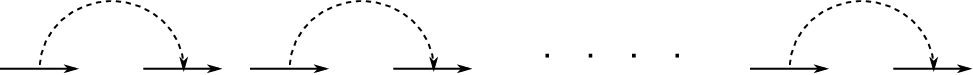}])=[[\includegraphics[height=.4cm, scale=.8]{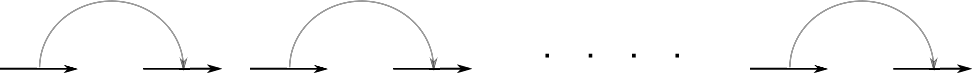}]],$$ 
where the left-hand side of the equation means the image of $n$ dashed arrows by $\varphi$ and the right-hand side of the equation means a Gauss diagram with a set of $n$ disjoint embedded pure virtual braids of degree $n$, that is, a set of $n$ embedded pure virtual braids of degree 1.
If $n=1$, $\varphi([ \includegraphics[height=.4cm, scale=.8]{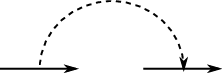} ])= [\includegraphics[height=.4cm, scale=.8]{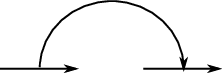}] - [\includegraphics[height=.07cm, scale=.2]{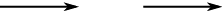}] = [[\includegraphics[height=.4cm, scale=.8]{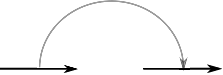}]]$.
Assume the formula holds less than or equal to $n$, it is easy to check that the formula holds $n+1$.
\\ (2) It suffices to show that $J_{n, l}(k) \subset J_{n,l+1}(k)$ for $n$, $l$ with $1 \leq l  \leq n-1$.
Let $[[D,H]] \in J_{n, l}(k)$. By assumption, there is an embedded pure virtual braid of degree $d$ in $H$, say to $(h_1,e_1)$, where $d \geq 2$.
Then $h_1$ can be represented by a pure virtual braid $[h_{1,1}, h_{1,2}] \cdot [h_{2,1}, h_{2,2}] \cdot \cdots \cdot [h_{m,1}, h_{m,2}]$ where deg$(h_{i,1})=1$ and deg$(h_{i,2})=d-1$ for any $i$. 
We define $h_1^j=h_{1,j} \otimes h_{1,j}^{-1} \otimes h_{2,j} \otimes h_{2,j}^{-1} \otimes \cdots \otimes h_{m,j} \otimes h_{m,j}^{-1}$ for $j=1,2$. 
Then deg$(h_1^1)=1$, deg$(h_1^2)=d-1$ and $D_0^{(h^j_1,e_1^j)} \sim D_0$, where $e^j_1$ is the restriction of $e_1$ to $h_1^j$. Therefore we have 

\begin{align*}
[[D_0, \{(h_1,e_1)\}]] 
%&=[e(h_1); [h_{1,1}, h_{1,2}] \cdot [h_{2,1}, h_{2,2}] \cdot \cdots \cdot [h_{m,1}, h_{m,2}] ]\\
&= [[D_0^{(h_1,e_1)}]] - [[D_0]] \\
&= [[D_0^{(h_1,e_1)}]] - [[D_0^{(h_1^1,e_1^1)}]] - [[D_0^{(h_1^2,e_1^2)}]] + [[D_0]] \\
&= [[D_0, \{(h_1^1,e_1^1), (h_1^2,e_1^2)\}]]. 
\end{align*}
%where an embedding of endpoints of each $h_{ij}$ by the embedding of $h_1$  and  ${h_1^j}$ ($j=1,2$) are the same as point wise.
Hence 
\begin{align*}
[[D,H]] &=[[D, \{(h_1,e_1), (h_2,e_2), \cdots , (h_l, e_l)\}]] \\
&= [[D, \{(h_1^1,e_1^1), (h_1^2,e_1^2), (h_2,e_2), \cdots , (h_l,e_l) \}]] \\
& \in J_{n,l+1}(k).
\end{align*}
%\\ (3) It suffices to show that $J_{n+1, l} \subset J_{n, l}$ for $n$, $l$ with $1 \leq l  \leq n$.
%Let $[D,H] \in J_{n+1, l}$. By assumption, there is a clasper of degree $d$ in $H$, say to $(h_1,e_1)$, where $d \geq 2$.
%Since $h_1 \in \Gamma_d(PV_k) \lhd \Gamma_{d-1}(PV_k)$ for some $k$, we have $[D,H] \in J_{n,l}$. 
\end{proof}

By Lemma \ref{Thm6.7}, we can redefine $J_n(k)$ as the ideal of $\mathbb{Z}\mathcal{VSL}(k)$ generated by elements $[[D, H]]$ where $D$ is any Gauss diagram on $k$ strands and $H$ is any set of disjoint embedded pure virtual braids for $D$ of degree $n$. 

\begin{proposition}\label{L_{n+1}V_n}%%6.8(H)
For any $n \geq 0$, if two virtual string links $L$ and $L'$ are $L_{n+1}$-equivalent, then $L$ and $L'$ are $V_{n}$-equivalent.
\end{proposition}

\begin{proof}
By Remark \ref{ringiso} and Lemma \ref{Thm6.7}, if two $k$-component virtual string links $L$ and $L'$ are $L_{n+1}$-equivalent,  then $L-L' \in J_{n+1,1}(k) \subset J_{n+1,n+1}(k) = J_{n+1}(k)$.
By the definition, it is equivalent to that $L$ and $L'$ are $V_{n}$-equivalent.
\end{proof}

%\begin{remark}
%A Gauss diagram formula $f= <\includegraphics[height=.4cm, scale=.8]{degree4fti.eps}, \cdot >$ is an invariant of degree 4 (see \cite{PV}, \cite{GPV}).
%Let $K_1 =  \includegraphics[height=.4cm, scale=.8]{degree2fti1.eps}$ and $K_2=\includegraphics[height=.4cm, scale=.8]{degree2fti2.eps}$. 
%Then, $f(K_1 \cdot K_2) = 1$ and $f(K_2 \cdot K_1) = 0$ and then $K_1 \cdot K_2$ is not $V_n$-equivalent to $K_2 \cdot K_1$ for $n \geq 4$.
%Therefore it follows from Proposition \ref{L_{n+1}V_n} that $K_1 \cdot K_2$ is not $L_{n}$-equivalent to $K_2 \cdot K_1$ for $n \geq 5$.
%Thus the group $\mathcal{VSL}/L_{l}$ ($l \geq 5$) is not abelian.  
%\end{remark}

%%%%%%%%%%%%%%%%%%%%%%%%%%%%%%%%%%%%%%%%%%%%%%%%%%%%%%%%%%%%%%%%%%%%%%
%%%%%%%%%%%%%%%%%%%%%%%%%%%%%%%%%%%%%%%%%%%%%%%%%%%%%%%%%%%%%%%%%%%%%%

%\noindent Department of Mathematics, Graduate School of Science, Osaka University \\
%1-10 Machikaneyama Toyonaka Osaka 560-0043 Japan \\
%Yuka Kotorii

\end{document}